\documentclass[12pt,reqno]{amsart}
\usepackage[mathscr]{eucal}
\usepackage{amssymb}
\usepackage{hyperref}
\usepackage{graphicx}

\def\ri{\mathrm{i}}
\def\C{{\mathbb C}}
\def\R{{\mathbb R}}
\def\CP{{\C}P}

\def\gam{\mbox{\raisebox{.45ex}{$\gamma$}}}

\renewcommand{\Re}{\operatorname{Re}}
\renewcommand{\Im}{\operatorname{Im}}

\def\vv{\mathsf{v}}
\newcommand\zbar{\overline{z}}

\def\scrN{\EuScript N}

\def\calQ{\mathcal Q}
\def\calP{\mathcal P}
\def\PP{\mathbb P}

\def\Fhat{\widehat F}
\def\ve{\mathsf{e}}

\def\re{\mathrm{e}}

\def\Impart{\operatorname{\sf Im}}
\def\Repart{\operatorname{\sf Re}}
\def\vec#1{\mathbf{#1}}

\newcommand{\sJ}{\mathsf{J}}
\newcommand\vphi{\boldsymbol{\phi}}
\newcommand\vpsi{\boldsymbol{\psi}}

\def\dx{\;dx}
\newcommand\bp{\mathsf p}

\newcommand{\dib}[1]{\dfrac{\partial}{\partial{#1}}}

\theoremstyle{definition}
\newtheorem{thm}{Theorem}
\newtheorem{prop}[thm]{Proposition}
\newtheorem{lemma}[thm]{Lemma}

\newtheorem*{defns}{Definitions}
\numberwithin{nex}{section}

\newtheorem*{urem}{Remark}

\begin{document}


\title[Geometric Realization of the YO Equations]{A Novel Geometric Realization \\ of the Yajima-Oikawa Equations}
\author{Annalisa Calini \& Thomas Ivey}
\address{Department of Mathematics, College of Charleston, Charleston, SC, USA}
\email{calinia@cofc.edu, iveyt@cofc.edu}
\begin{abstract}
{We show that the Yajima-Oikawa (YO) equations, a model of short wave-long wave interaction, arise from a simple geometric flow on  curves in the 3-dimensional sphere $S^3$ that are transverse to the standard contact structure. For the family of periodic plane wave solutions of the YO equations studied by Wright, we construct the associated transverse curves, derive their closure condition, and exhibit several examples with non-trivial topology.}
\end{abstract}
\keywords{Geometric evolution equations; integrable systems; contact structures; transverse curves; long wave-short wave models}
\maketitle

\section{Introduction}

This work is part of our investigation of curve flows in the 3-sphere $S^3$ that are invariant under the action of the group $SU(2,1)$ of pseudoconformal transformations, which preserves the standard contact structure on the sphere. While the focus of our previous study~\cite{CI2021} was Legendrian curves in the 3-sphere and geometric flows for such curves which are integrable (i.e., inducing integrable evolution equations for their fundamental differential invariants), in this note we discuss an interesting connection between an integrable model of short wave-long wave interaction and a geometric flow for curves that are transverse to the contact structure.\smallskip

The pseudoconformal geometry of $S^3$ is inherited from the geometry of the space $\C^3$ endowed with the indefinite Hermitian form
\begin{equation}
\label{sesqform}
\langle \vec{z}, \vec{w}\rangle := \ri (z_3 \overline{w_1} -z_1  \overline{w_3})+ z_2 \overline{w_2}.
\end{equation}
Given the standard action of $SL(3,\C)$  on $\C^3$, let $SU(2,1)$ denote the subgroup that preserves this form.
Let $\scrN \subset \C^3$ be the {\em null cone}, i.e., the set of nonzero null vectors for \eqref{sesqform}.
The set of complex lines on the null cone is diffeomorphic to $S^3$, the unit sphere in $\C^2$ (see \eqref{S3toz}).
It follows that the linear action of $SU(2,1)$ on $\C^3$ induces an action on $S^3$ known as the group of pseudoconformal transformations.
We will let $\pi$ denote the complex projectivization map from $\C^3$ minus the origin to $\CP^2$, as well as its restriction to the null cone, giving a commutative diagram:
\begin{center}
\begin{picture}(20,30)
\put(0,30){\makebox(0,0){$\scrN$}}
\put(15,30){\makebox(0,0){$\subset$}}
\put(2,14){\scriptsize{$\pi$}}
\put(25,26){$\C^3\backslash \{0\}$}
\put(0,22){\vector(0,-1){15}}
\put(0,0){\makebox(0,0){$S^3$}}
\put(15,-2){\makebox(0,0){$\subset$}}
\put(37,14){\scriptsize{$\pi$}}
\put(36,-1){\makebox(0,0){$\C P^2$}}
\put(35,22){\vector(0,-1){15}}
\end{picture}
\end{center}
The pseudoconformal action preserves the standard contact structure on $S^3$, defined for curves in $S^3$ in terms of
their lifts relative to $\pi$ as follows.

\begin{defns} Let $\gam:I \to S^3$ be a regular parametrized curve on an interval $I \subset \R$.
Then $\gam$ is {\em Legendrian} if it has a lift $\Gamma:I \to \scrN$
satisfying
\begin{equation}\label{contact-cond}
\Im \langle \Gamma_x, \Gamma \rangle = 0.
\end{equation}
By contrast, $\gam$ is a {\em transverse curve} or {\em T-curve} if its lift satisfies
\begin{equation}\label{transv-cond}
\Im \langle \Gamma_x, \Gamma \rangle \not= 0.
\end{equation}
In other words, the tangent vector of a T-curve is everywhere transverse to the contact planes.
(Note that both conditions \eqref{contact-cond},\eqref{transv-cond} are invariant under a change of lift,
i.e., multiplying $\Gamma$ by a nonzero complex-valued function.)
\end{defns}

Let $\gam:I \to S^3$ be a regular curve with lift $\Gamma:I \to \scrN$.  Then $\Gamma$ and its derivative $\Gamma_x$ satisfy
$\langle \Gamma, \Gamma \rangle=0$ and $\Re \langle \Gamma_x, \Gamma \rangle=0.$
If $\gam$ is a transverse curve then $\Im \langle \Gamma_x, \Gamma \rangle \not= 0$, so we can assume the normalization $\langle \ri \Gamma, \Gamma_x \rangle =1$; we can furthermore choose a lift that
also satisfies $\langle \Gamma_x, \Gamma_x \rangle=0$ (see \S2 for more details). With these assumptions, we define a geometric
flow based on the second derivative $\Gamma_{xx}$ as
\begin{equation}
\label{skewnormal}
\Gamma_t=\ri \left( \Gamma_{xx} - \langle \Gamma_{xx}, \Gamma_{x} \rangle \ri \Gamma \right),
\end{equation}
which induces a well-defined flow for the T-curve $\gam = \pi \circ \Gamma$.
Note that the vector field in parentheses on the right-hand side is a modification $\Gamma_{xx}$ that lies in $\{ \Gamma_x\}^\perp$, the orthogonal complement of $\{ \Gamma_x\}$. If we let $\bp_{ \{ \Gamma_x\}^\perp}$ denote the orthogonal projection onto $\{\Gamma_x\}^\perp$,
then writing~\eqref{skewnormal} as $\Gamma_t=\ri \bp_{ \{ \Gamma_x\}^\perp}\left(\Gamma_{xx}\right)$
suggests an analogy with the vortex filament flow $\gam_t=\gam_x \times \gam_{xx}$ (or binormal flow) for an arc length parametrized curve $\gam$ in Euclidean space~\cite{LP1991}, with the skew-symmetric operator $\ri \bp_{ \{ \Gamma_x\}^\perp}$ the analogue of the symplectic operator $T_x \times$ for the binormal flow.

In Sections 2--4 we construct adapted frames for transverse curves---both local frames (akin to the Frenet frames of Euclidean geometry) and non-local `natural' frames---and show that equation~\eqref{skewnormal} can be rewritten in terms of a convenient non-local adapted frame $(\Gamma, \Gamma_x, B)$ as
\[
\Gamma_t= \ri z B,
\]
where $z$ is a complex curvature, part of the set $(z, m)$, $z\in \mathbb{C}, m\in \mathbb{R}$, of geometric invariants of $\Gamma$.
After deriving the evolution equations for the geometric  invariants induced by a general vector field on (lifts of) transverse curves, we show that the evolution induced by~\eqref{skewnormal} on the invariants $(z, m)$ is the following system of nonlinear PDE
\[
\begin{split}
z_t & = \ri (z_{xx} - mz),\\
m_t & = 2(|z|^2)_x,
\end{split}
\]
known as the Yajima-Oikawa (YO) or Long-Wave-Short-Wave equations, a completely integrable model of interaction of long and short waves.

In Section 5, we use the connection between the Lax pair for the YO equations at given $(z, m)$ and the adapted frame of the associated transverse curve to construct examples of geometric realizations of solutions of the YO equations.  We focus on the family of plane wave solutions studied by Wright in~\cite{W2006}, derive closure conditions for the associated curves, and construct explicit formulas. The plane wave solutions, though simple at the YO level, provide a wealth of closed transverse curves with non-trivial topology. We present visualizations of several examples, that illustrate how the knot type and the geometry relate to the parameters in the YO solutions.

In Section 6 we discuss some open questions and directions for future work.

\section{Pseudoconformal Frames and Curvature}

Let $\gam:I \to S^3$ be a T-curve, and $\Gamma$ be a lift satisfying $\Im \langle \Gamma_x, \Gamma \rangle >0$. Since the restriction of the Hermitian form \eqref{sesqform} to the complex span $\mathcal{S}=\text{span}_\mathbb{C}\{\Gamma, \Gamma_x\}$ is non-degenerate, we construct a smooth adapted frame by selecting two linearly independent null vectors--$\Gamma$ itself and a second vector $V\in \mathcal{S}$--and adding a third vector $B$ which is spacelike and spans the complex line orthogonal to $\mathcal{S}$.

As described in \cite{CI2021}, the ordered triple $(\Gamma, B, V)$ of vectors in $\C^3$ can be chosen to satisfy the following inner product relations
\[
\begin{split}
\langle \Gamma, \Gamma \rangle = \langle V, V \rangle = \langle B, \Gamma \rangle = \langle B, V \rangle =0, \\
\langle\Gamma, V \rangle = -\ri, \quad \langle V, \Gamma \rangle = \ri, \quad \langle B, B \rangle = 1.
\end{split}
\]
as well as the condition $\operatorname{det} (\Gamma, B, V)=1$ (meaning that the vectors form the
columns of a unimodular matrix).
We call a triple that satisfies these relations a \emph{unimodular null frame}.  In the rest of this section we describe how a smoothly-varying unimodular null frame, including the lift $\Gamma$ as its first member, can be chosen in an essentially unique way for a regular T-curve, allowing us to identify fundamental invariants.

\subsection*{Local frame.}
In \cite{CI2021} it is shown that, under suitable nondegeneracy assumptions, any parametrized T-curve $\gam$ has a unimodular null frame field $(\Gamma, B, V)$, constructed in terms of algebraic functions of the components of $\gam$ and its derivatives, that satisfies
\begin{equation}\label{Lframe}
\dfrac{d \Fhat}{dx}  = \Fhat \begin{bmatrix}\tfrac13 \ri p & -\ri q & m \\ 0 & -\tfrac23 \ri p & q \\ 1 & 0 & \tfrac13 \ri p \end{bmatrix}.
\end{equation}
where $\Fhat$ denotes the matrix with columns $\Gamma, B, V$, and $m,p,q$ are real-valued {\em fundamental differential invariants} of the parametrized curve.
We refer to this as the {\em local frame}, and it is unique up to multiplication of each column
by the same cube root of unity.  It is the analogue of the (local) Frenet frame for a unit-speed curve $\gam:\R \to \R^3$ in Euclidean space.

\subsection*{Natural frame.} In the Euclidean case, one can also construct the (non-local) {\em relatively parallel} or {\em natural} frame $(T, U_1, U_2)$, where $U_1 = \cos\theta\, N + \sin \theta\, B$, and $U_2 = -\sin \theta\, N + \cos \theta\, B$, with $\theta = - \int \tau\,dx$ and $N$ and  $B$ the unit normal and binormal vectors. This natural frame, which is unique up to a choice of antiderivative $\theta$, satisfies
$$\dfrac{dT}{dx} = k_1 U_1 + k_2 U_2, \quad \dfrac{dU_1}{dx} = -k_1 T, \quad \dfrac{dU_2}{dx} = -k_2 T,$$
so that the normal vectors $U_1, U_2$ rotate only in the direction of the tangent line.  The
functions $k_1=k\cos (\theta)$ and $k_2=k\sin (\theta)$ are {\em natural curvatures} \cite{B1975}.

By analogy with the Euclidean case, given the local frame $\Fhat$ for a T-curve $\gam$ we can use an antiderivative to neutralize the rotation of normal vector $B$ in the normal plane, forming a new unimodular null frame field defined by
\[ F= \Fhat \exp (\theta \sJ), \qquad \text{where } \theta = -\int p\,dx, \quad
\sJ = \begin{pmatrix} \tfrac13 \ri & 0 & 0 \\ 0 & -\tfrac23 \ri & 0 \\ 0 & 0 & \tfrac13 \ri \end{pmatrix}.\]
It follows that $F$ satisfies the {\em nonlocal frame equations}
\begin{equation}\label{NLframe}
\dfrac{d F}{dx}  = F \begin{bmatrix}  0 & -\ri \zbar & m \\ 0 & 0 & z\\ 1 & 0  &0 \end{bmatrix},
\end{equation}
where $z = z= e^{\ri \theta} q$ and $m$ is the same as in \eqref{Lframe}.
One can interpret $z$ as a complex curvature, measuring how
the tangent line $\PP\{\Gamma, V\}$ bends within the complex projective plane.
The real-valued invariant $m=\Impart \langle V, V_x \rangle $  measures the deviation of the projectivization of $V$ from being a Legendrian curve in $S^3$.

\subsection*{Companion $\lambda$-frames}
Any two unimodular null frames at the same point of $S^3$ are linked by a transformation of the following form (see, e.g., \cite{CI2021})
\begin{equation}\label{CoF}
\tilde \Gamma = \nu \Gamma, \quad
\tilde B = \dfrac{\overline{\nu}}{\nu}\left( B +\mu \Gamma \right), \quad
\tilde V = {\overline{\nu}}^{-1} \left[V -\ri \overline{\mu}B- (\lambda + \tfrac12 \ri |\mu|^2 )\Gamma\right],
\end{equation}
where $\nu,\mu$ are complex, with $\nu \ne 0$ and $\lambda$ is real.

Given the local frame for a T-curve, we modify the frame using $\mu=0$, $\nu=1$ and $\lambda$ constant in \eqref{CoF}, to obtain the {\em companion $\lambda$-frame} $\tilde \Gamma=\Gamma, \tilde B=B, \tilde V=V-\lambda \Gamma$.  This modified frame satisfies
\[
\tilde \Gamma_x = \left(\frac{1}{3} \ri p+\lambda\right)\tilde \Gamma +\tilde V,  \quad
\tilde B_x =-\ri q  \tilde \Gamma -\frac{2}{3} \ri p \tilde B, \quad
\tilde V_x =(m - \lambda^2) \tilde \Gamma+(\frac{1}{3} \ri p-\lambda) \tilde V.
\]
If we make a similar modification to a natural frame for a T-curve we obtain the {\em companion $\lambda$-natural frame}, which satisfies
\begin{equation}\label{lambdaNLframe}
\dfrac{d \tilde F}{dx}  = \tilde F \begin{bmatrix}  \lambda & -\ri \zbar & m-\lambda^2 \\ 0 & 0 & z\\ 1 & 0  &-\lambda \end{bmatrix}.
\end{equation}

\begin{urem}
Note that if the projection of the frame vector $\tilde V$ of a companion $\lambda$-frame is a Legendrian
curve in $S^3$ (so that $m-\lambda^2=0$) then the same is true for the companion frame constructed
using $-\lambda$.
\end{urem}

\section{Curve Flows and the Yajima-Oikawa Equations}\label{YOconnection}
If $\gam(x,t)$ is a smooth variation of a T-curve  and $(\Gamma, B, V)$ is a smoothly-varying choice of natural frame, then the vector field $\Gamma_t = f \Gamma + g B + h V$ must satisfy
\begin{equation}\label{nonstretch}
h_x=-2\Repart f \quad \text{and} (\Impart f)_x = \Repart( g \overline{z})
\end{equation}
in order to keep the frame adapted, as shown in Proposition 11 of \cite{CI2021}.  It follows that the nonlocal invariants $m$ and $z=k + \ri \ell$ evolve by
\begin{equation}
\label{YOhierarchy}
\begin{bmatrix} k \\ \ell \\ m\end{bmatrix}_t  = \calP \begin{bmatrix} \Impart g \\ -\Repart g \\ \tfrac12 h \end{bmatrix},
\end{equation}
where
\[
\calP=\begin{pmatrix}
-3\ell D^{-1}\circ \ell & 3\ell D^{-1} \circ k - D^2 + m & 2 D\circ k + k D \\
3 k D^{-1} \circ \ell +D^2 - m & -3k D^{-1}\circ k      & 2 D\circ \ell + \ell D \\
2 k D + D\circ k &  2\ell D + D\circ \ell  & 2(m D + D \circ m) - D^3
\end{pmatrix}
\]
and $D=\partial_x$. The matrix operator $\calP$ is skew-adjoint, and forms a Hamiltonian pair with
\[
\calQ =\small{\begin{pmatrix} 0 & 1 & 0 \\ -1 & 0 & 0 \\ 0 & 0 & D \end{pmatrix}}.
\]

In particular, if $\Gamma$ evolves by
\begin{equation}\label{Bflow}
\Gamma_t = \ri z B,
\end{equation} then the invariants $z$ and $m$ satisfy the \emph{Yajima-Oikawa} (YO) equations
\begin{equation}\label{YO}
\begin{aligned}
z_t & = \ri (z_{xx} - mz),\\
m_t & = 2(|z|^2)_x.
\end{aligned}
\end{equation}
The YO system first appeared in work by Grimshaw in the context of internal gravity waves \cite{G1975}, and was derived by Yajima \& Oikawa \cite{YO1976} and by Djordjevic \& Redekopp \cite{DR1977}  as an integrable model of interaction of a long wave (of amplitude $m$) and a short wave (of complex amplitude $z$).
\subsection*{Integrability} The YO system~\eqref{YO} is the compatibility condition of the following pair of linear systems
\begin{equation}\label{LS}
\vphi_x=U\vphi, \quad \vphi_t=V\vphi,
\end{equation}
where 
\begin{equation}\label{LaxYO}
U=\begin{bmatrix} \lambda & 0 & 1\\ \ri z & 0 & 0 \\ m & \zbar & -\lambda \end{bmatrix}, \qquad
V=\begin{bmatrix} -\tfrac13 \ri \lambda^2  &  -\ri \zbar   & 0 \\ \lambda z- z_x & \tfrac23 \ri \lambda^2  & z \\ |z|^2 & \ri(\lambda \zbar-\zbar_x)& -\tfrac13 \ri \lambda^2 \end{bmatrix},
\end{equation}
with eigenfunction $\phi \in \C^3$, and spectral parameter $\lambda \in \C$.  When $\lambda\in \mathbb{R}$, $U$ and $V$ take value in the Lie algebra $\mathfrak{su}(2,1)$ of the subgroup of $SL(3, \mathbb{C})$ that preserves the Hermitian form \eqref{sesqform}.
Taking the transpose of \eqref{LaxYO} and complex conjugating, we obtain
\begin{equation}\label{LFrenet}
F_x = F \begin{bmatrix} \lambda & -\ri \zbar & m \\ 0 & 0 & z \\ 1 & 0 & -\lambda \end{bmatrix},
\quad
F_t = F \begin{bmatrix} \tfrac13 \ri \lambda^2  &  \lambda \zbar-\zbar_x   & |z|^2 \\ \ri z & -\tfrac23 \ri \lambda^2  & \ri (z_x-\lambda z) \\ 0 & \zbar & \tfrac13 \ri \lambda^2 \end{bmatrix}.
\end{equation}
Comparing the first of these equations to \eqref{lambdaNLframe} shows that system \ref{LFrenet} can be interpreted as the Frenet equations for the companion natural $\lambda$-frame of an T-curve with curvatures $z$ and $\tilde{m}=m+\lambda^2$, and which evolves by the flow
\begin{equation}\label{flowwithlambda}
\Gamma_t=\ri z B +\tfrac{1}{3} \ri \lambda^2 \Gamma.
\end{equation}
This connection between the YO Lax pair and the evolution of (framed) curves allows the construction of interesting examples of transverse curves associated with simple solutions of the YO system, as shown in the rest of the article.

\bigskip

\section{Plane wave solutions}

In~\cite{W2006}, Wright investigates the linear stability of plane wave solutions of the YO equations~\eqref{YOW} and derives explicit solutions of the associated Lax pair in order to construct homoclinic orbits of unstable plane waves.

\newcommand{\WU}{\boldsymbol{\mathrm U}} 
\newcommand{\WV}{\boldsymbol{\mathrm V}} 
\newcommand{\FU}{U} 
\newcommand{\FV}{V} 
\newcommand{\vecr}{\mathsf r}
\subsection{Equivalent Versions of YO}
In~\cite{W2006} the YO system is given as
\begin{equation}\label{YOW}
\begin{aligned}
 A_\tau &= -2\ri(A_{xx} -A B),\\
 B_\tau &= 4(|A|^2)_x
\end{aligned}
\end{equation}
for complex $A(x,\tau)$ and real $B(x,\tau)$, and its Lax pair is given as
\begin{equation}\label{LSW}
\vpsi_x = \WU \vpsi, \qquad \vpsi_\tau = 2\WV \vpsi
\end{equation}
with
\[\WU = \begin{bmatrix} \ri \zeta & A & \ri B \\ 0 & 0 & -\overline{A} \\ -\ri & 0 & -\ri\zeta \end{bmatrix},\quad
\WV = \begin{bmatrix} \tfrac13 \ri\zeta^2 & \zeta A -\ri A_x & \ri |A|^2 \\ 2\overline{A} & -\tfrac23 \ri \zeta^2 & \zeta \overline{A} -\ri \overline{A_x} \\
					0 & -A & \tfrac13 \ri\zeta^2 \end{bmatrix},\]
where $\zeta$ and $\tau$ denote Wright's spectral parameter and time variable respectively.
(Wright's YO equations include an extra parameter which we omit because it can be removed by a simple change of variable.)
The equations \eqref{YOW} are equivalent to \eqref{YO} under the substitutions $A=\overline{z}$, $B=m$ and $\tau=\tfrac12 t$.  Moreover, the linear systems \eqref{LSW} and \eqref{LS} are equivalent under a change of gauge, since with these substitutions, $\WU = M U M^{-1}$ and $\WV = M V M^{-1}$,
where
\[M = \begin{bmatrix} 0 & 0 & 1 \\ 0 & 1 & 0 \\ -\ri & 0 & 0 \end{bmatrix}.\]

\subsection{Wright's Solutions}
In this section,  we will present Wright's solutions, rewritten in terms of our variables.
We will then make use of the eigenfunction formulas in~\cite{W2006}, appropriately adapted to the geometric framework, to construct the associated transverse curves.  We will then identify the
parameter choices that give rise to closed transverse curves.

\begin{prop}[\cite{W2006}]\label{wrightprop}  For real constants $a,b,k$ and $\Lambda$ such that $a>0$, the functions
\[
z(x,t)=a \re^{-\ri N}
\qquad m(x,t)=b, \quad \text{where } N := k x- \Lambda t
\]
give a solution of \eqref{YO} if and only if the {\em dispersion relation} $b+k^2 + \Lambda = 0$ is satisfied.
When these $z,m$ are substituted into \eqref{LS}, a non-trivial solution of \eqref{LS} is given by
\[
\vphi(x,\tau) = \re^{\ri (\mu x + \nu t)} P \vecr,
\]
where
\[ P =\begin{bmatrix} 1 & 0 & 0\\ 0 & \re^{-\ri N} & 0 \\ 0 & 0 & 1 \end{bmatrix},
\]
and $\vecr$ is a nonzero common eigenvector of the matrices
\begin{equation}\label{commons}
\begin{bmatrix} \lambda & 0 & 1 \\ \ri a & \ri k & 0 \\ b & a & -\lambda \end{bmatrix}, \qquad
\begin{bmatrix} -\tfrac13 \ri \lambda^2  &  -\ri a   & 0 \\ a (\lambda +\ri k) & \tfrac23 \ri \lambda^2  -\ri \Lambda & a\\ a^2 & a (k+\ri\lambda)& -\tfrac13 \ri \lambda^2 \end{bmatrix}
\end{equation}
with eigenvalues $\ri \mu$ and $\ri \nu$, respectively.
\end{prop}

\bigskip
It is easy to check that the matrices in \eqref{commons} have a non-trivial common eigenvector if and only if $\mu$ and $\nu$ satisfy
\begin{subequations}
\begin{gather}
(\mu^2+b+\lambda^2)(\mu-k)+a^2=0, \label{mdeqnmu} \\
\nu = \mu^2 -k^2 -\Lambda + \tfrac23 \lambda^2.\label{mdeqnnu}
\end{gather}
\end{subequations}
In order to construct a fundamental matrix solution for \eqref{LS} using solutions of this form, let $\mu=m_1,m_2,m_3$ be three
distinct roots of \eqref{mdeqnmu} and let $n_1, n_2, n_3$ be the corresponding values of $\nu$ given by substituting these
into \eqref{mdeqnnu}.
Then a matrix solution is given by $\Phi = P R E$, where
\[ R = \begin{pmatrix} -1& -1 & -1 \\
\dfrac{a}{k-m_1} & \dfrac{a}{k-m_2} & \dfrac{a}{k-m_2}\\[8pt]
\lambda-\ri m_1 & \lambda-\ri m_2 & \lambda-\ri m_3\end{pmatrix},
\quad E = \begin{pmatrix} \re^{\ri (m_1 x + n_1 t)}  & 0 & 0 \\
 0 & \re^{\ri(m_2 x + n_2 t)} & 0 \\
 0 & 0 & \re^{\ri(m_3 x + n_3 t)} \end{pmatrix}.
\]

Our discussion in \S\ref{YOconnection} implies that if $\Phi$ is a fundamental matrix solution of the YO Lax pair for a real value of $\lambda$, and taking value in the group $SU(2,1)$, then $F = \Phi^\dagger$ is a $\lambda$-natural frame matrix for a transverse curve evolving by \eqref{flowwithlambda}.
Since $\lambda  \in \R$ implies that the coefficient matrices in \eqref{LS} take value in $\mathfrak{su}(2,1)$, we can
ensure that our fundamental matrix takes value in $SU(2,1)$ by modifying it to be equal to the identity matrix when $x=t=0$:
\begin{equation}\label{normalizedPhi}
\Phi = P R E R^{-1} P_0^{-1},\qquad \text{where } P_0 = P\big\vert_{x=t=0}.
\end{equation}
Using this matrix to construct the natural frame, and taking the projectivization of the first frame vector $\Gamma$ to
obtain a transverse curve $\gam$, we now consider the question of when the resulting curve is smoothly closed.

\begin{prop} Suppose the fundamental matrix $\Phi = P R E$ described above corresponds to a $\lambda$-natural frame for a T-curve $\gam$.  Then $\gam$ is closed of length $L$ if and only if there is a cube root $\omega$ of unity such that
\begin{equation}\label{closurem}
\re^{\ri m_j L} = \overline{\omega} \re^{\ri k L/3}.
\end{equation}
\end{prop}

\begin{proof}  Let $F$ be a natural $\lambda$-frame along the curve, satisfying the spatial part of \eqref{LFrenet}, and let $\Fhat$ be the local frame related to $F$ by $\Fhat = F \exp(-\theta \sJ)$.  Because the local frame is uniquely determined by derivatives of $\gam$, up to multiplying by a cube root of unity, then $\gam$ is closed of length $L$ if and only if
\begin{equation}\label{localclosurerel}
\Fhat(x+L) = \omega \Fhat(x).
\end{equation}
For the solutions of Prop. \ref{wrightprop}, $\theta=\arg z = -N$; rewriting \eqref{localclosurerel} in terms of $F=\Phi^\dagger$, then in terms of the factors
of $\Phi$ given by \eqref{normalizedPhi}, and simplifying (using the fact that $\sJ$ commutes with $M$ and $P$) gives the condition \eqref{closurem}.
\end{proof}

Without loss of generality, we will take $L=2\pi$ from now on, and assume the roots of \eqref{mdeqnmu}
are numbered in ascending order, i.e., $m_1 < m_2 < m_3$.  Note that, for a given value of $\lambda$, these roots determine the coefficients in the polynomial via:
\begin{equation}\label{parmrels}
\begin{aligned}
k &= m_1+m_2+m_3, \\
a^2 &= (k-m_1)(k-m_2)(k-m_3),\\
b &= m_1 m_2 + m_1 m_3 + m_2 m_3-\lambda^2.
\end{aligned}
\end{equation}


\begin{lemma}\label{closurelemma}
The above closure conditions \eqref{closurem} are satisfied if and only if there are positive integers $p,q$ such that
$k$ satisfies either
\begin{subequations}\label{kcases}
\begin{gather}
-\tfrac12 (2p+q) < k < \tfrac12(p-q)\label{firstkcase}\\
\intertext{or}
 k > \tfrac12(p+2q). \label{secondkcase}
\end{gather}
\end{subequations}
In either case, the roots are given by
\begin{equation}\label{mvalues}
m_1 = \tfrac13(-2p-q+k), \quad m_2 = \tfrac13(p-q+k), \quad m_3 = \tfrac13 (p+2q+k),
\end{equation}
and $\omega = \re^{2\pi \ri \epsilon/3}$ where $\epsilon = 0,1,2$ is such that $3m_j-k \equiv \epsilon$ modulo 3.
\end{lemma}
\begin{proof}
We can rewrite the closure condition~\eqref{closurem} as
\begin{equation}\label{mtoell}
m_j=l_j + \tfrac{1}{3}\epsilon + \tfrac13 k,
\end{equation}
for some integers $l_1<l_2<l_3$. The second relation in~\eqref{parmrels} is satisfied for $a>0$ if and only if
\begin{equation}
\label{mcases}
m_1<k<m_2<m_3 \qquad \text{or} \qquad m_1<m_2<m_3<k.
\end{equation}
When written in terms of the positive integers $p=l_2-l_1$ and $q=l_3-l_2$, the two conditions in~\eqref{mcases} become those in~ \eqref{kcases}.

Conversely, suppose a pair positive integers $p, q$ satisfy either of the equations in \eqref{kcases} for some real number $k$.  Let $\epsilon=0,1,2$ be chosen
so that $p-q \equiv \epsilon$ modulo 3, and let
\[l_1=\tfrac13 (-2p-q-\epsilon), \qquad l_2=\tfrac13 (p-q-\epsilon), \quad l_3=\tfrac13 (p+2q)-\frac13\epsilon.\]
Then with $m_j$ given by \eqref{mtoell}, condition~\eqref{mcases} is satisfied.
\end{proof}

\subsection{Visualizing Examples}
In this section we will exhibit examples of closed transverse curves in $S^3$, generated using the fundamental matrix $\Phi$ corresponding to
Wright's solutions, with closure conditions imposed using Lemma \ref{closurelemma}.  In particular, we will observe knotted transverse curves which, because their differential invariants are the same at each time, move by rigid motion under the flow \eqref{flowwithlambda}.

In more detail, given two positive integers $p,q$ one may select any value of $\lambda$ and a real value of $k$ satisfying one of the
inequalities in \eqref{kcases}.  (The other parameters involved in the solution are determined by equations \eqref{parmrels} and \eqref{mvalues}.)  This yields two distinct 2-parameter families of closed curves for each pair $(p,q)$.  (Exactly how we construct these curves is explained below.)  We will assume that $p,q$ are relatively prime; experiments indicate that the knot types
are the same when $p,q$ are multiplied by a common integer factor.

In the case \eqref{firstkcase} we observe that the curve in $\R^3$ is a right-handed $(p,q+p)$-torus knot.  Recall that the type of a $(m,n)$-torus knot
depends only on the {\em unordered pair} $\{m,n\}$.  However, for our examples we find that
when $k$ is close to its lower limit, the knot takes a shape with $p$ strands that wind along the torus the long way, while when $k$ is close to its upper limit the knot has $p+q$ strands winding the long way.
In general, the knot shape is more compact and symmetric when $\lambda=0$; Figure~\ref{RHanders}
shows two shapes for the same $p,q,k$ but different $\lambda$ values.  (Note that \eqref{parmrels} shows that these curves have
the same differential invariant $z$ but different constant values $m=b$.)

We showed in \cite{CI2021} that transverse curves for which $z=0$ identically are $SU(2,1)$-congruent to curves which run along the circular fibers of the Hopf fibration.  Thus, when $k$ approaches one of the roots $m_j$, $a=|z|$ will approach zero and the closed curve will approach a multiply-covered circle congruent to a Hopf fiber.
In Figure~\ref{hom25}, we show a family of right-handed $(2,5)$-torus knots, corresponding to a range of
$k$-values, where at both ends of the family the curve approaches a multiply-covered circle.

In the case \eqref{secondkcase} we observe that the curve in $\R^3$ is a left-handed $(p,q)$-torus knot.  (When $p=1$ or $q=1$ the curve is unknotted, as shown in Figure~\ref{LH11withV}.)  When $k$ is close to its lower limit the curve has $q$ strands winding around the torus the long way, and its shape approaches a circle covered $q$ times.
For large values of $k$, the curve approaches a flattened teardrop shape,
with the knot crossings compressed into a small region near where $x=\pi$.
Both these limiting behaviors are illustrated in Figure~\ref{LH23}.


\begin{figure}
\includegraphics[width=4.5cm]{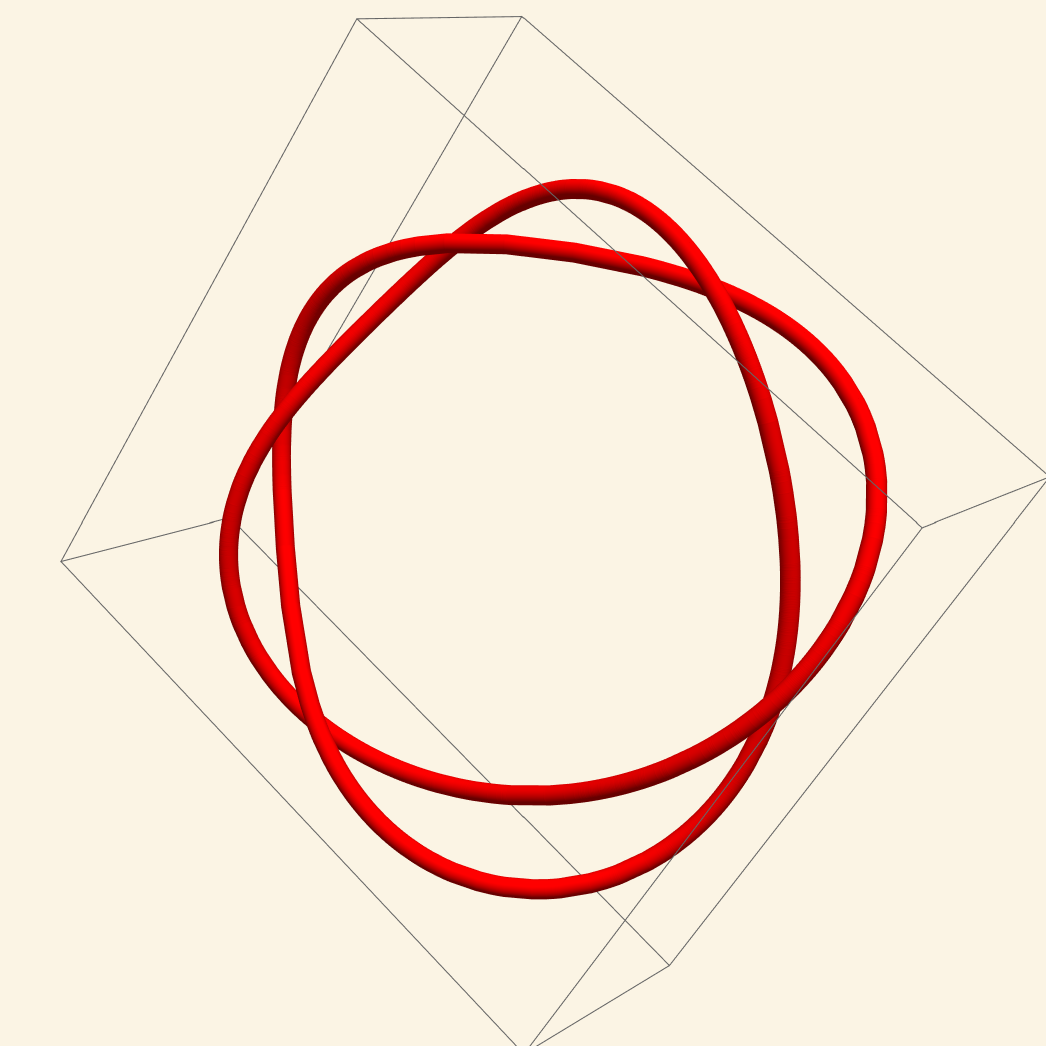}
\includegraphics[width=4.5cm]{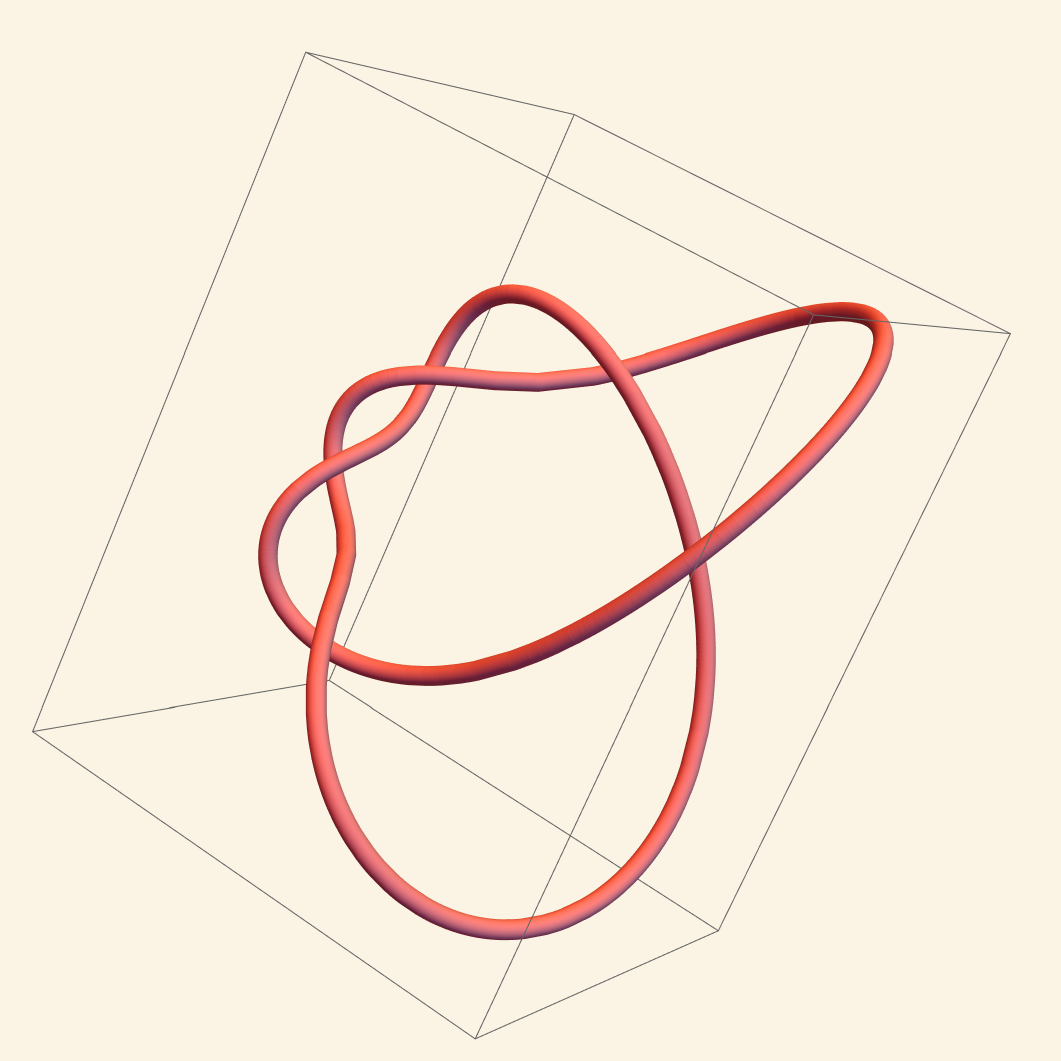}
\includegraphics[width=4.5cm]{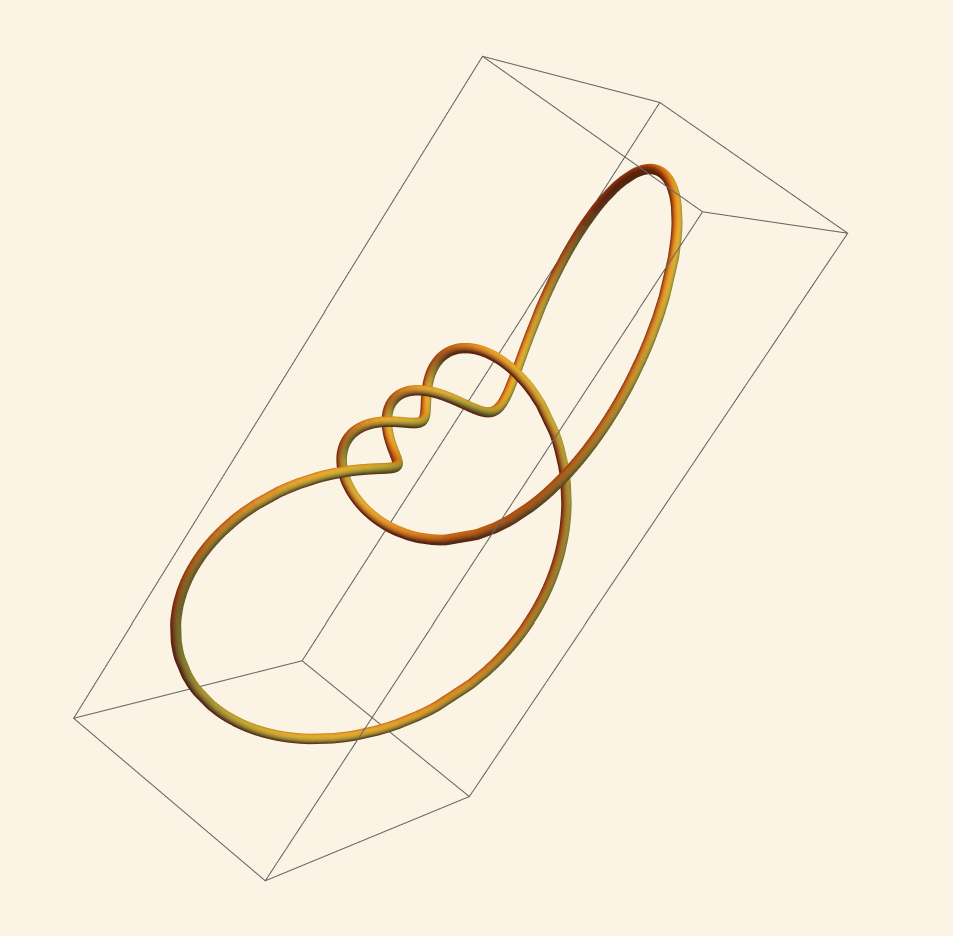}\\
\includegraphics[width=4.5cm]{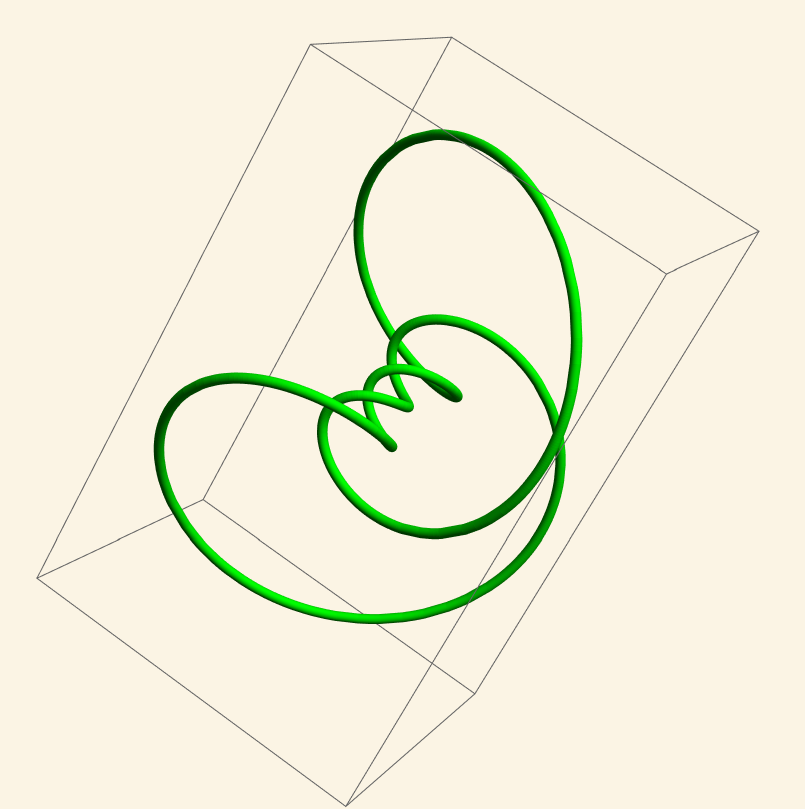}
\includegraphics[width=4.5cm]{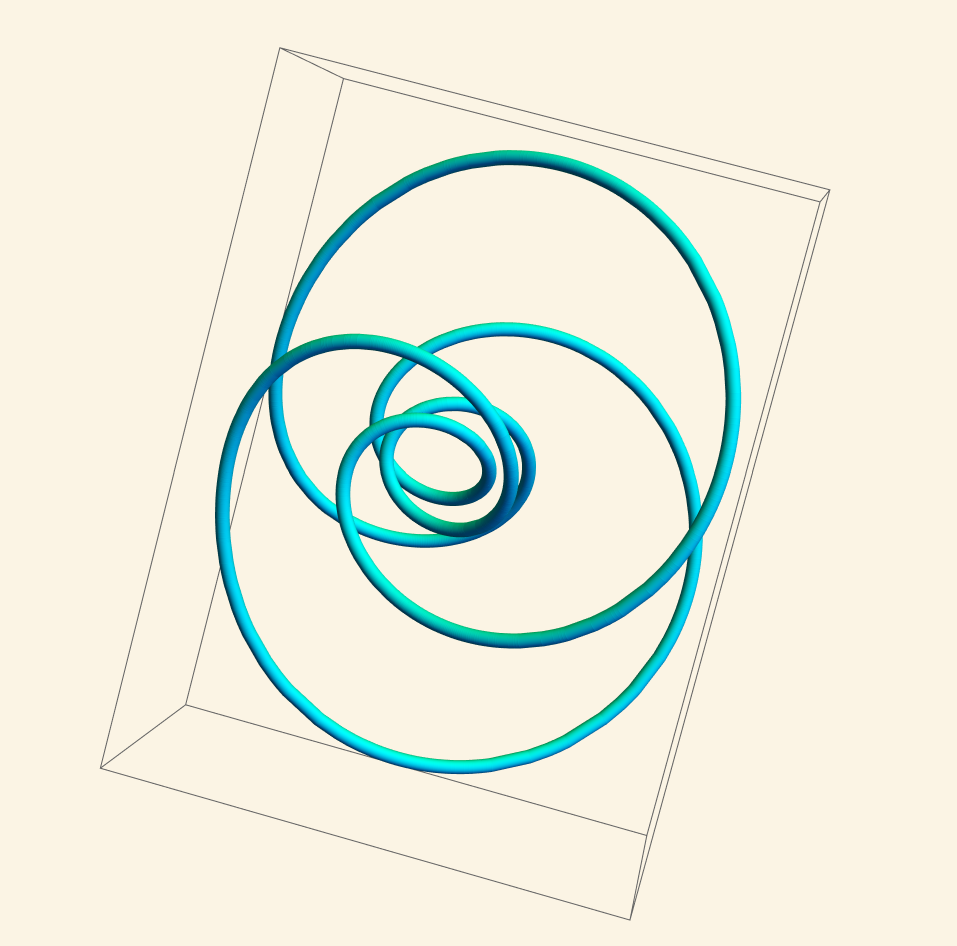}
\includegraphics[width=4.5cm]{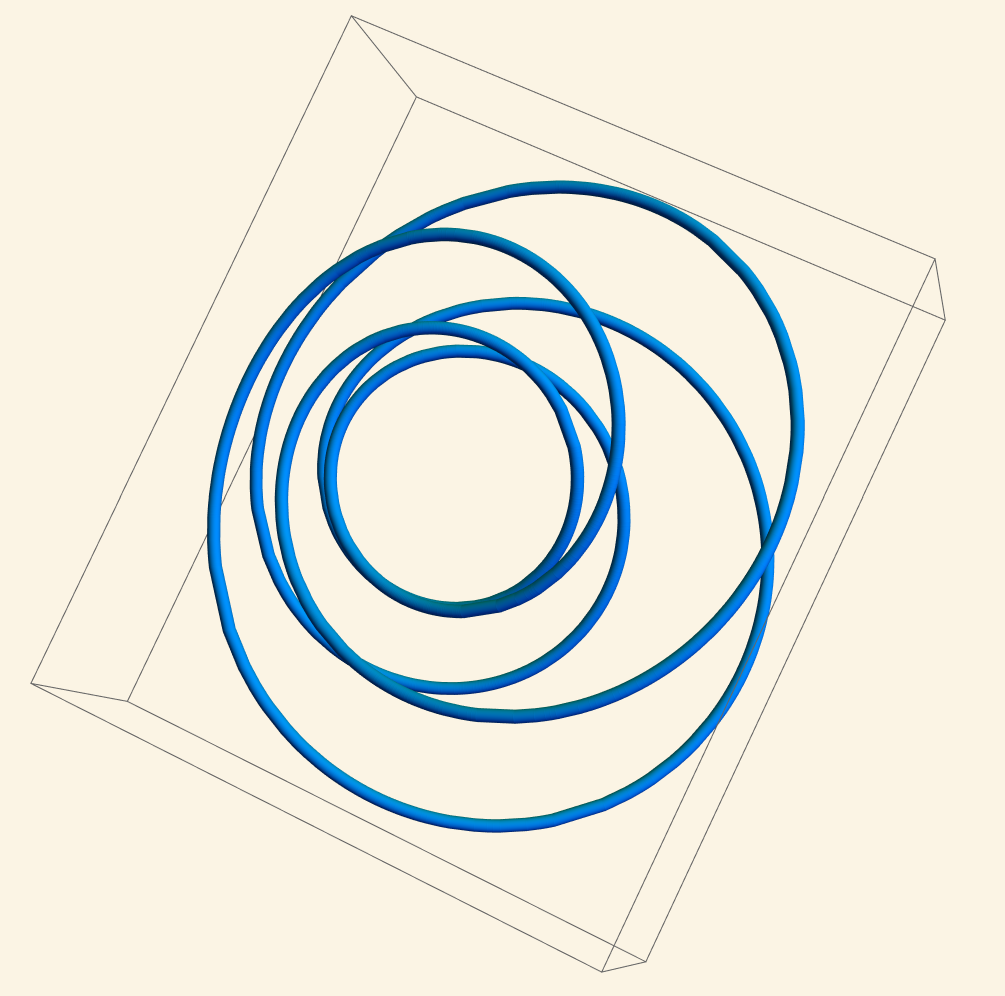}
\caption{A family of $(2,5)$ torus knots obtained using $p=3, q=2$, $\lambda=0$ and the values
$k=-3.85, -3.25, -2.5$ in the top row and $k=-1.75, -0.7, 0.2$ in the bottom row.}
\label{hom25}
\end{figure}


\begin{figure}
\includegraphics[width=5cm]{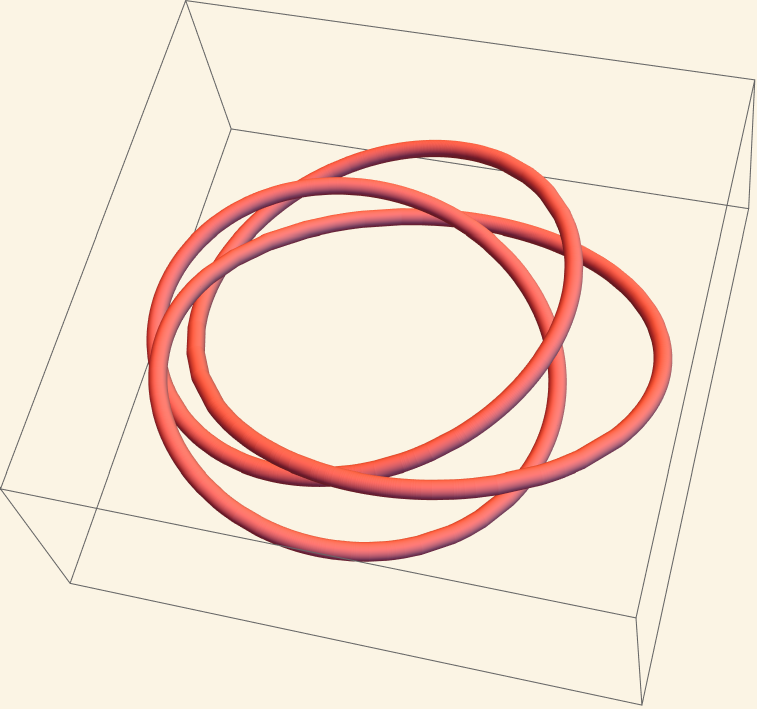}\qquad
\includegraphics[width=5cm]{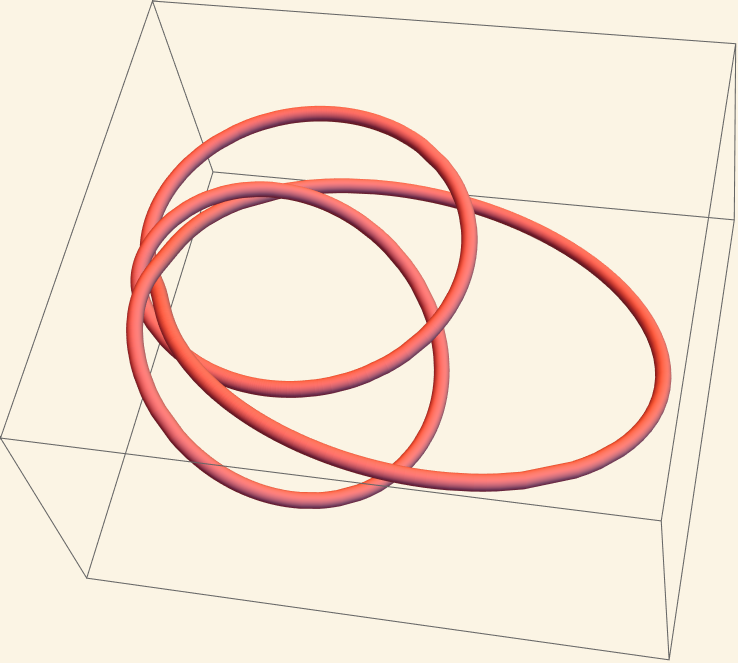}
\caption{Right-handed $(3,4)$ torus knots obtained using $p=1$, $q=3$ and $k=-2.2$; on the left
$\lambda=0$, while on the right $\lambda=3.1$.}
\label{RHanders}
\end{figure}

\begin{figure}
\includegraphics[width=5.3cm]{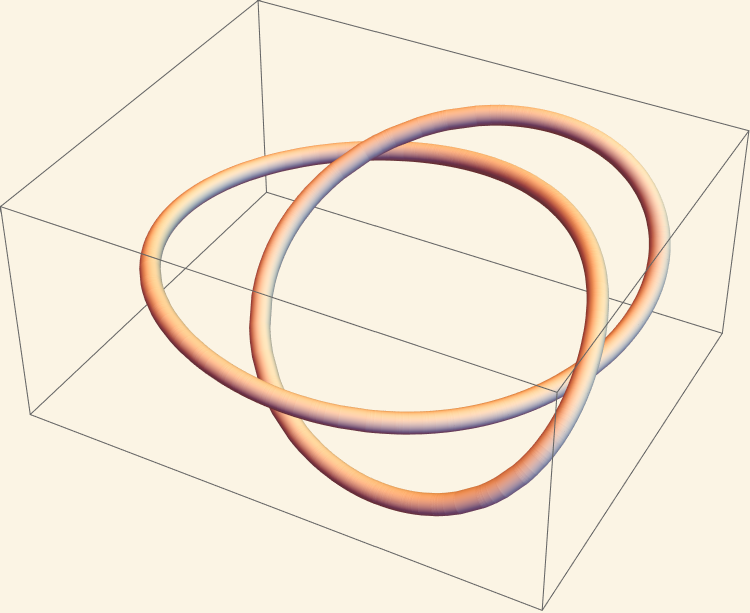}\qquad
\includegraphics[width=6cm]{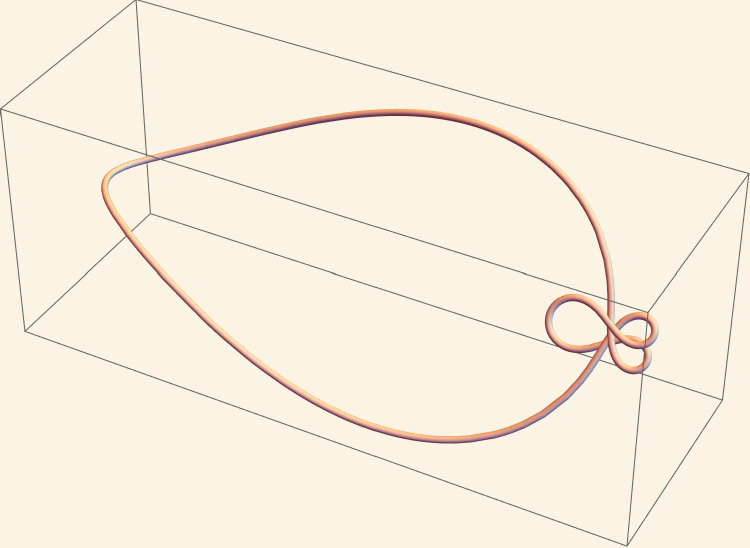}
\caption{Left-handed $(2,3)$ torus knots (i.e., trefoils) obtained using $p=1$, $q=2$ and $\lambda=0$; on the left
$k=4.6$, while on the right $k=31$.}
\label{LH23}
\end{figure}

\subsection{Constructing Transverse Curves}
Once we have a fundamental matrix solution $\Phi$ for the linear system \eqref{LS}, the first component of the $\lambda$-natural frame is then given by
$$\Gamma = F \ve_1  = \Phi^\dagger \ve_1,$$
taking value in the null cone $\scrN$.
We produce curves in $S^3$ using a projection $\pi: \scrN \to S^3$ given in terms of the components of $\Gamma$ by
\begin{equation}\label{S3toz}
z_1 = \dfrac{\Gamma_3 - \ri \Gamma_1}{\Gamma_3 + \ri \Gamma_1}, \qquad z_2 = \dfrac{\sqrt{2}, \Gamma_2}{\Gamma_3 + \ri \Gamma_1},
\end{equation}
where $(z_1, z_2)$ lie on the unit sphere in $\C^2$ equipped with its standard Hermitian inner product.  For purposes of visualization, we in turn apply stereographic projection into $\R^3$ (using the point $z_1=0$, $z_2=\ri$ as pole) given by
\[ \sigma: (z_1, z_2) \mapsto \left( \dfrac{\Re z_1}{1-\Im z_2}, \dfrac{\Im z_1}{1-\Im z_2}, \dfrac{\Re z_2}{1-\Im z_2}\right).\]

\begin{urem} The action of $SU(2,1)$ on the null cone preserves the 1-form $\alpha_N = (dg,g)$, which is the pullback
under $\pi$ of the standard contact form on $S^3$, given by $\alpha_S = \tfrac12 \Im ( \zbar_1 dz_1 + \zbar_2 dz_2)$.
The contact planes in $S^3$ annihilated by this 1-form are orthogonal to the Hopf fibers.  Since $S^3$ is parallelizable, we can choose an globally defined orthogonal frame $(\vv_0, \vv_1, \vv_2)$ such that $\vv_1, \vv_2$ are tangent to the contact planes.  For purposes
of visualizing the contact distribution, we will use the following vectors in $\R^3$ which are tangent to the image of this distribution
under stereographic projection:
\begin{align*}
\sigma_* \vv_1 &= -(z+xy)\dib{x} + \tfrac12(x^2-y^2+z^2-1)\dib{y} +(x-yz) \dib{z},\\
\sigma_* \vv_2 &= \tfrac12(x^2-y^2-z^2+1) \dib{x} +(xy-z)\dib{y}+(xz+y)\dib{z}
\end{align*}
Figure~\ref{LH11withV} shows how the curve $\gam$ is transverse to the planes spanned by these vector fields.

Recall from \eqref{NLframe} that when $m=0$ the frame vector $V$ projects to a Legendrian curve in $S^3$.  Figure~\ref{LH11withV} also
shows this companion curve which in this example is linked with $\gam$ and tangent to the contact planes.
\end{urem}

\begin{figure}[h!]
\includegraphics[trim=50 30 50 50,clip=true,width=9cm]{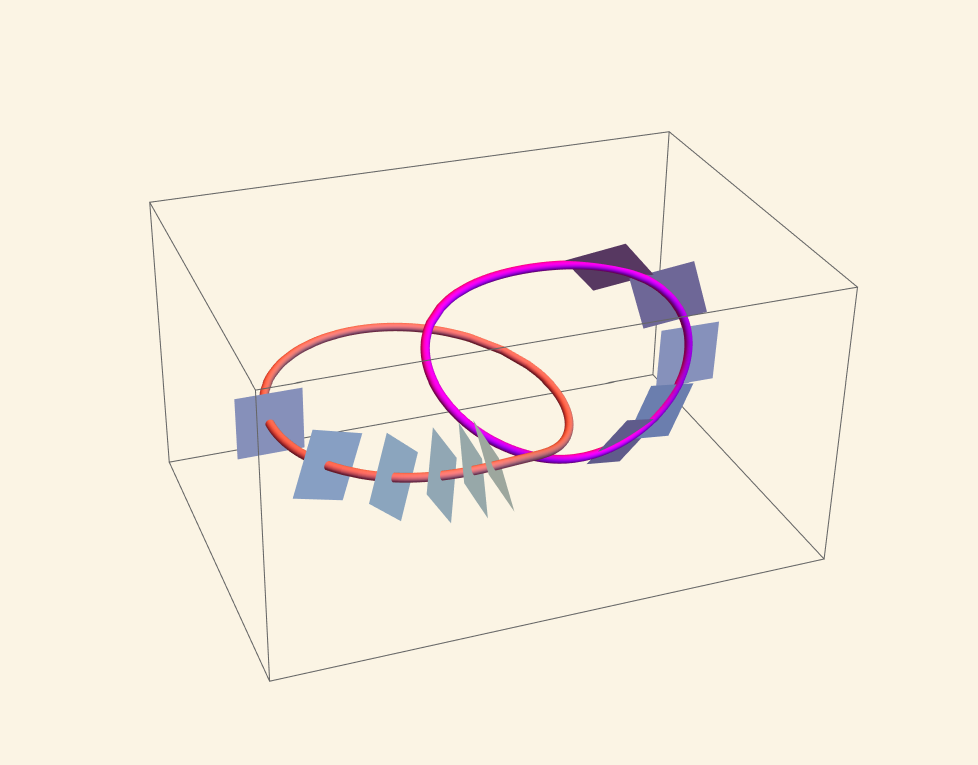}
\caption{At left, in orange, is an unknotted curve $\gam$ generated using parameter values $p=q=1$, $k=2$ and $\lambda=1/\sqrt{3}$.  Substituting these values into \eqref{mvalues} and \eqref{parmrels} shows that $m=b=0$, hence the curve traced by the projectivization
of frame vector $V$ (shown at right in magenta) is a Legendrian curve.  Along both curves we have drawn some planes of the contact distribution.}
\label{LH11withV}
\end{figure}

\section{Discussion}

We have shown how the YO equations arise, somewhat unexpectedly, from a simple geometric flow for curves in $S^3$ that are transverse to the standard contact structure. The recent renewed interest in the YO equations and related systems (see, e.g., ~\cite{C-HDLS2021,C-HDLlS2022, LG2022}), the analogies between the  geometric flow considered in this work and the vortex filament flow, and the relatively simple reconstruction of the transverse curve in terms of solutions of the YO Lax pair, makes this a good case for exploring questions such as recursion schemes and the geometric and topological properties of transverse curves related to special solutions of the YO equations.

A natural direction of investigation is the study of the integrable hierarchy of vector fields for transverse curves associated with the YO hierarchy.  These are generated by beginning with a conserved density for the YO equations, e.g.,
\[\rho_1 = \tfrac12 m, \quad \rho_2 = \tfrac12 |z|^2, \quad \rho_3 = \tfrac12 \Im(\zbar z_x) - \tfrac18 m^2, \quad
\rho_4 = -\tfrac12 \left( m |z|^2 + |z_x|^2\right), \ldots \]
and forming the vector field $X_n = f_n \Gamma + g_n B + h_n V$ where $(\Gamma, B, V)$ is a natural frame.  The coefficients are determined by the corresponding density
as follows.
As in \eqref{YOhierarchy} write $z = k+\ri \ell$ and express the density $\rho_n$ in terms of real invariants $k,\ell,m$ and their $x$-derivatives.  Let
\[ (a_n, b_n, c_n)^T = {\mathsf E} \rho_n \]
where $\mathsf E$ denotes the vector-valued Euler operator.  Then the components of $X_n$ are $h_n = 2c_n$, $g_n = \ri (a_n + \ri b_n)$ and
$f_n = -(c_n)_x + \ri d_n$ where $d_n = \int \Re(g_n \zbar) \dx$.   (Thus, these vector fields satisfy the conditions in \eqref{nonstretch}
to preserve the adapted frame.)

The first few vector fields generated this way are
\[
\begin{split}
X_1 & =  V = \Gamma_x, \\
X_2 & =  \ri z B = \ri (\Gamma_{xx}-m \Gamma),\\
X_3 & =  \left(\tfrac14 m_x + \tfrac12 \ri |z|^2\right) \Gamma + z_x B - \tfrac12 m V\\
X_4 & =  \left(\tfrac12 |z|^2_x - \ri \Im(\overline{z} z_x)\right)\Gamma +\ri (z_{xx}-m z)B -|z|^2 V,\\
\end{split}
\]

The fact that the antiderivative $d_n$
 is always expressible in terms of $z,m$ and their derivatives is somewhat mysterious.
However, we observe that
these antiderivatives are expressible in terms of Hermitian inner products of the vector fields themselves:
\[
d_{2j} = -\tfrac12 \sum_{k=1}^{2j-2} \langle X_{2j-k}, X_{1+k}\rangle, \qquad
		d_{2j+1} = -\tfrac12 \sum_{k=1}^{2j-1} (-1)^k \langle X_{2j+1-k}, X_{1+k}\rangle.
\]
Since $d_n = \Re \langle X_n, V \rangle = \tfrac12 (\langle X_n, V\rangle +\langle V, X_n\rangle)$ and $V=X_1$, these identities are equivalent to
\[
\sum_{k=0}^{2j-1} \langle X_{2j-k}, X_{1+k}\rangle = 0 \qquad
\text{and} \qquad
\sum_{k=0}^{2j} (-1)^k \langle X_{2j+1-k}, X_{1+k}\rangle = 0.
\]
These show a remarkable parallel with the situation for vector fields in the hierarchy for the vortex filament flow~\cite{L1999}, where the antiderivative required for the tangential component of $X_{n+1}$ is expressible in terms of inner products of the vector fields up to $X_n$.  In that case, the analogous identities were proved using the first-order `geometric' recursion operator for the vector fields.  In our case, it may be sufficient to have a second-order recursion operator that relates $X_{n+2}$ to $X_n$.

\bigskip

B\"acklund and Darboux transformations as well as Miura transformations are other common features of integrable systems. In particular, the classical B\"acklund transformation for the sine-Gordon equation has its origins in relating pair of pseudospherical surfaces through line congruences (see, e.g., \S7.5 in \cite{IL2016}). It is possible that an analogous transformation exists between T-curves evolving by the YO flow \eqref{Bflow}; one might expect that
 the curves would be joined by a congruence of circles in $S^3$ expressed in terms of the vectors of the natural frame.

In relation to a potential Miura transformation,  one can investigate the evolution equations induced by \eqref{Bflow} for the {\em tangent indicatrix}, i.e., the curve
in $S^3$ traced out by the projectivization of the frame vector $V$.
It is natural to ask how the invariants of these indicatrices are related to those of the primary curve, and furthermore whether,
when the primary curve evolves by an integrable geometric flow, the invariants of the indicatrix evolve by a related integrable system.


\begin{thebibliography}{99}
\def\book{\it}
\def\art{\sf}
\def\jou{\sl}
\newcommand{\jvol}[1]{{\bf{#1}}}
\newcommand{\ditto}{{\leavevmode\vrule height 2pt depth -1.6pt width 23pt\,}}

\bibitem{B1975} R. L.~Bishop, {\art There is more than one way to frame a curve}, American Math. Monthly, 822:46--51, 1975.

\bibitem{CI2021} A.~Calini, T.~Ivey, {\art Integrable geometric flows for curves in pseudoconformal $S^3$},
Journal of Geometry and Physics, 166, 104249, 2021.

\bibitem{C-HDLS2021} M.~Caso-Huerta, A.~Degasperis, S.~Lombardo, M.~Sommacal, {\art A new integrable model of long wave–short wave interaction and linear stability spectra}, Proc. R. Soc. A. 477 (2021) 20210408.

\bibitem{C-HDLlS2022} M.~Caso-Huerta, A.~Degasperis, S.~Lombardo, M.~Sommacal, {\art Periodic and solitary wave solutions of the long wave– short wave Yajima-Oikawa-Newell model}, Fluids 7 (7) (2022) 227. 

\bibitem{DR1977} V.~Djordjevic, L.~Redekopp, {\art On two-dimensional packets of capillary-gravity waves},
Journal of Fluid Mechanics, 79(4), 703-714, 1977.

\bibitem{G1975} R.H.J. Grimshaw, {\art The modulation and stability of an internal gravity wave}, Res. Rep. School Math. Sci., Univ. Melbourne, no. 32-1975.

\bibitem{LG2022} R.~Li, X.~Geng, {\art Periodic-background solutions for the Yajima–Oikawa long-wave–short-wave equation}, Nonlinear Dyn. 109 (2) (2022) 1053– 1067.

\bibitem{IL2016}  T.A. Ivey, J.M. Landsberg, {\book Cartan for Beginners: Differential Geometry via Moving Frames and Exterior Differential Systems} (2nd ed.), Graduate Studies in Mathematics {\bf 175}, American Mathematical Society, 2016

\bibitem{L1999} J.~Langer, {\art Recursion in Curve Geometry}, New York J. Math. 5, 25–51 (1991).

\bibitem{LP1991} J.~Langer, R.~Perline, {\art Poisson geometry of the filament equation}, J Nonlinear Sci 1, 71–93 (1991).


\bibitem{W2006} O. C.~Wright, {\art Homoclinic Connections of Unstable Plane Waves of the Long-Wave--Short-Wave Equations}, Studies in Applied Mathematics 117:71--93, 2006.

\bibitem{YO1976} N.~Yajima, M.~Oikawa, {\art Formation and Interaction of Sonic-Langmuir Solitons}, Progress of Theoretical Physics, 56(6), 1719--1739, 1976.



\end{thebibliography}
\end{document}